\documentclass[10pt]{article}
\usepackage{amsfonts,amssymb,amsmath}
\usepackage{graphicx}
\usepackage{url}
\usepackage[colorlinks]{hyperref}
\usepackage[usenames,dvipsnames]{color}
\usepackage{amsthm}
\hypersetup{urlcolor=BlueViolet}
\hypersetup{linkcolor=black}
\hypersetup{citecolor=blue}

\graphicspath{{figs/}}

\newcommand{\R}{\mathbb{R}}

\newcommand{\qb}{{\bf q}}

\newtheorem{prop}{Proposition}
\newtheorem{lemma}{Lemma}
\newtheorem{theo}{Theorem}
\newtheorem{conj}{Conjecture}
\newtheorem{definition}{Definition}

\newtheoremstyle{PropositionNum}
        {\topsep}{\topsep}              
        {\itshape}                      
        {}                              
        {\bfseries}                     
        {.}                             
        { }                             
        {\thmname{#1}\thmnote{ \bfseries #3}}
\theoremstyle{PropositionNum}
\newtheorem{propn}{Proposition}

\begin{document}

\title{On central configurations of twisted crowns}
\author{E. Barrab\'es\footnote{esther.barrabes(at)udg.edu, U. de Girona} \& J.M. Cors\footnote{cors(at)epsem.upc.edu, U. Polit\`ecnica de Catalunya}}
\date{June 22, 2015}



\maketitle

{\bf Abstract}

We consider the planar central configurations of the Newtonian $\kappa n$-body problem consisting in $\kappa$ groups of $n$-gons where all $n$ bodies in each group have the same mass, called $(\kappa, n)$-crown. We study the location and the number of central configurations when $\kappa=2$.  For $n=3$ the number of central configurations varies depending on the mass ratio, whereas for $n\geq 4$ the number is at least three.  We also prove that for $n\geq 3$  there always exist three disjoint regions where the configuration can be located. Finally, we study which $(\kappa, n)$-crowns are convex.

\section{Introduction}

In the $n$-body problem a configuration is {\it central} if the acceleration
vector for each body is a common scalar multiple of its position vector (with
respect to the center of mass). The study of central configurations allows to obtain explicit solutions of the $n$-body problem where the shape remains constant up to rescaling and rotation. While much is known about specific cases, usually
involving symmetry or assuming that some bodies are infinitesimally small, less is known
about the general structure of the set of central configurations. See \cite{Saari05} for a introduction to the subject.

There are numerous results in the literature demonstrating
the existence of planar central configurations with specific symmetries. The problem consisting in $p$ groups of regular $n$-gons where all $n$ bodies in each group have the same mass have been studied by several authors. In \cite{Corbera2010}, the existence of central configurations is proved when the masses are at the vertices of $p$ homothetic ({\it nested}) regular $n$-gons for all $p\geq2$ and $n\geq2$. In \cite{Llibre2009}, specific cases when $p=3,4$ are studied. 
In \cite{2015ZhaoChen} the existence of central configurations
is proved in the $(pn + gn)$-body problem. In that case, $p$ regular $n$-gons are
nested, and $g$ regular $n$-gons are rotated by $\pi/n$ ({\it twisted})
compared to the first $p$ $n$-gons.  Recently \cite{2015Montaldi}  describe a uniform proof of all previous existence
results, using well-known arguments for the existence of symmetric solutions to variational
problems.

There also exist general results in the case of two regular $n$-gons. In \cite{YuZhang2012} the authors prove that a necessary condition for a central configuration is that the angle between the two $n$-gons have to be $0$ or $\pi/n$, that is, nested or twisted. In \cite{2015YuZhang}  the authors prove that a necessary condition for a central configuration is that the two twisted $n$-gons must have the same number of bodies. 

In the nested case of two $n$-gons, in \cite{Moeckel1995}, the authors prove that for every mass ratio, there are exactly two
planar central configurations. As the masses are varied, spatial central configurations appear through bifurcation from
planar ones. In particular, spatial configurations can be found which are
arbitrarily close to being planar. Other works on nested two $n$-gons are \cite{Zhang2001}, \cite{Zhang2002}.





In this paper we will focus our study to central configurations of twisted $n$-gons, that we will call {\it twisted crowns} (see Definition 1). 
Although the paper is focussed mainly in the case of twisted crowns with two $n$-gons, we include here the derivation of the general equations for central configurations of $\kappa$ $n$-gons, each one with $n$ bodies of the same mass (see Section 2). Some results when $\kappa > 2$ will be presented in a forthcoming paper. 

The main goal is to count the number of central configurations of two twisted $n$-gons. Specifically, when $n=3$ we give the exact number, that varies between one and three, whereas for $n\geq 4$ we prove there always exist at least three. Moreover, we describe the set of admissible places where the two twisted $n$-gons can be located (see Section 3).
Finally, in Section 4 we study which twisted crowns are convex. 

The paper include an Appendix where the detailed proof of some technical results are given.


\section{Definitions and equations}
Consider the planar Newtonian $\kappa n$-body problem consisting in $\kappa$ groups of $n$ bodies where all $n$ bodies in the $j$-th group have equal mass $m_j$, $j=1,\ldots,\kappa$.
Let $\qb_{ji}\in \R^2$, $i=1,\ldots,n$, $j=1,\ldots,\kappa$, be the position of each body, in a reference frame where the center of mass is at the origin of coordinates. It is well known (see \cite{Saari05}) that a \emph{central configuration} of the $\kappa n$-body problem is a configuration $\qb=(\qb_{11},\qb_{12},\ldots,\qb_{\kappa n})\in \R^{2\kappa n}$ such that, for a value of $\lambda\in \R$, satisfies the equation
\begin{equation}
 \nabla U(\qb) + \lambda M \qb =0,
 \label{eq:cc}
\end{equation}
where $U$ is the Newtonian potential
$$ U(\qb) =
\sum_{j=1}^{\kappa} \sum_{i=1}^n \left(
\sum_{l=i+1}^n \dfrac{m_j^2 }{||\qb_{ji}-\qb_{jl}||} + \sum_{l=j+1}^n  \sum_{\nu=1}^n \dfrac{m_j m_l}{||\qb_{ji}-\qb_{l\nu}||}
\right),
$$
and $M$ is the diagonal matrix with diagonal $m_1,\ldots,m_1,m_2,\ldots,m_2,\ldots,m_{\kappa},\ldots,m_{\kappa}$ (each mass $m_j$ repeated $n$ times).

We are interested in central configurations such that all the bodies in the same group form a regular $n$-gon (we also
 call the groups {\it rings}).

\begin{definition}
A central configuration formed by $\kappa$ groups of $n$ bodies in a regular $n$-gon, such that all the masses of the same group are equal, is called a \emph{crown of $\kappa$ rings} or a $(\kappa,n)$--crown.
\end{definition}

We call $\qb_{j1}$ the \emph{leader} of the group, so once its position is known, all the others bodies in the same group are fixed. Then introducing polar coordinates, we have that
\begin{equation}
\qb_{j1} = a_j e^{i\varpi_j},\; \qb_{jk}=\qb_{j1}e^{i2\pi(k-1)/n}, \quad k=2,\ldots,n, \; j=1,\ldots,\kappa,
\label{eq:positions}
\end{equation}
where $\varpi_j$ and $a_j$ are, respectively, the polar angle of the leader and the radius of the $j$-th ring. Without lost of generality  $\varpi_j\in(-\pi/n,\pi/n]$,
$j=1,\ldots,\kappa$. 

It could seem natural to consider an increasing sequence of radius, so that the first ring is in a circle smaller than the second one, and so on. Nevertheless, we will not make any assumption on the ``order'' of the radii of the rings, that is, there is no order established beforehand on the sequence $(a_j)_{j=1,\ldots,\kappa}$. We will explain our choice with more detail later on.

The system of equations given in~(\ref{eq:cc}) consists in $2\kappa n$ equations but, due to the symmetries of the problem, only $2\kappa$ are different. The equations related to the leaders are
$$ \dfrac{\partial U}{\partial \qb_{j1}} + \lambda m_j \qb_{j1} =0,\quad  j=1,\ldots,\kappa.$$
Equivalently, using (\ref{eq:positions}), we get
\begin{equation}
\begin{split}
\dfrac{m_j}{a_j^2}  \sum_{k=1}^{n-1}  \dfrac{e^{i2\pi k/n}-1}{(2-2\cos(2k\pi/n))^{3/2}}  & + \\
\sum_{\underset{l\neq j}{l=1}}^{\kappa} m_l \sum_{k=1}^n &
\dfrac{a_l e^{i(\varpi_l-\varpi_j+2\pi k/n)}-a_j}{(a_l^2+a_j^2-2a_la_j\cos(\varpi_l-\varpi_j+2\pi k/n))^{3/2}} +\lambda a_j=0,
\end{split}
\label{eq:cc3}
\end{equation}
for $ j=1,\ldots,\kappa$.

We have that both, the real and the imaginary parts of each equation in (\ref{eq:cc3}), must be zero.
Since the imaginary part of the first sum in~(\ref{eq:cc3}) is always zero,
then we have to impose that the imaginary part of the second sum in~(\ref{eq:cc3}) must be zero, that is,
\begin{equation}
\sum_{\underset{l\neq j}{l=1}}^{\kappa} m_l \sum_{k=1}^n
\dfrac{a_l \sin (\varpi_l-\varpi_j+2\pi k/n)}{(a_l^2+a_j^2-2a_la_j\cos(\varpi_l-\varpi_j+2\pi k/n))^{3/2}}=0, \quad j=1,\ldots,\kappa.
\label{eq:i}
\end{equation}

The difference $\varpi_l-\varpi_j$ represents how much one ring is rotated with respect to the other. Notice that, if
for all pairs $j$ and $l$, $j\neq l$, $|\varpi_l-\varpi_j|$ is 0 or $\pi/n$, then all $\kappa$ equations in~(\ref{eq:i})
are satisfied.

\begin{definition}
Consider the  $j$-th and $k$-th rings of a $(\kappa,n)$-crown. When $\varpi_j-\varpi_k=0$, the two rings are \emph{nested}. When  $|\varpi_j-\varpi_k|=\pi/n$, the two rings are \emph{twisted}. A $(\kappa,n)$--crown that contains at least two twisted rings is called a twisted crown, whereas if all the rings are nested, is called a nested crown.
\end{definition}

Yu and Zang (\cite{YuZhang2012}) show that a central configuration formed by two $n$-gons with equal masses in each ring must be nested or twisted. That is, the only $(2,n)$-crowns are those nested or twisted.

From now on we only consider nested or twisted crowns. Without lost of generality, in a nested crown we can take $\varpi_j=0$ for all $j=1,\ldots,\kappa$. In a twisted crown we always can take $\varpi_1=0$. Then, when $\kappa=2$ the only possibility for $\varpi_2$ is $\pi/n$.
When $\kappa=3$, there is also only one possibility: one of the rings is rotated an angle of $\pi/n$ radians with respect to the other two. As we have not consider any order in the radii, it is enough to take $\varpi_2=\pi/n$ and $\varpi_3=0$. Depending on the relative values of the radii $a_1,a_2,a_3$ the rotated ring will be the inner one, the one in the middle, or the outer one. In general, when $\kappa\geq 4$, the different scenarios can be obtained as follows: for just one rotated ring, take $\varpi_2=\pi/n$ and $\varpi_3=\ldots=\varpi_{\kappa}=0$, for two rotated rings, take $\varpi_2=\varpi_3=\pi/n$ and $\varpi_4=\ldots=\varpi_{\kappa}=0$, and so on. See Figure~\ref{fig:twistedcrowns}.

\begin{figure}[!ht]
\centering
\includegraphics[scale=0.6]{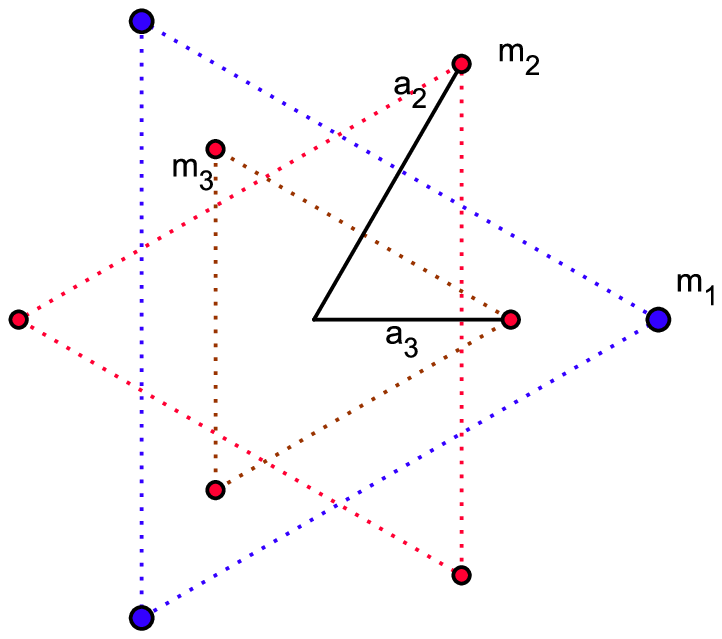}
\includegraphics[scale=0.58]{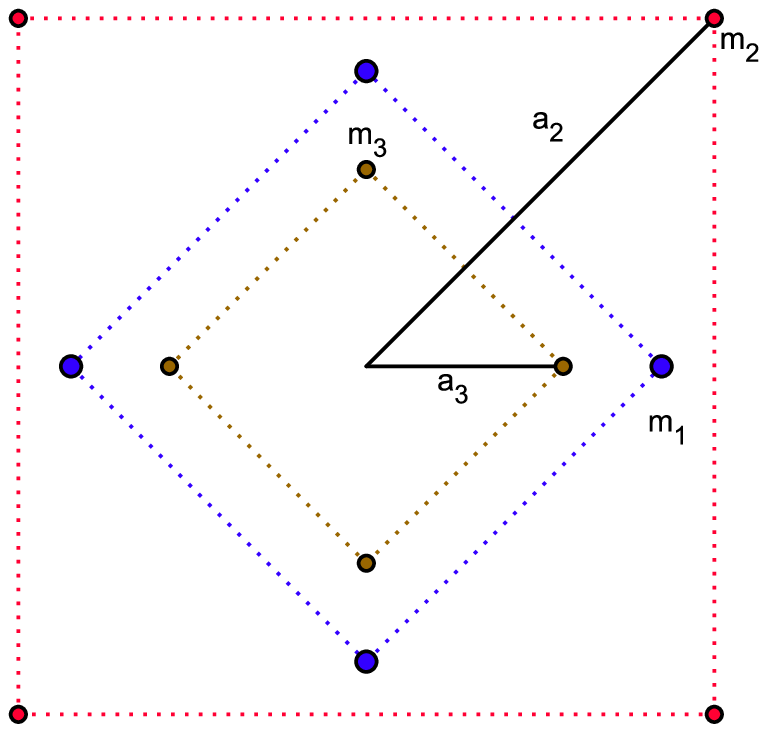}
\caption{Two examples of twisted crowns: left, a $(3,3)$-crown; right, a $(3,4)$-crown. In both we have taken $a_1=1$}
\label{fig:twistedcrowns}
\end{figure}

Therefore, a $(\kappa,n)$-crown is obtained if the real part of System of equations (\ref{eq:cc3}) is given by
\begin{equation}
-\dfrac{m_j}{a_j^2} S_n -
\sum_{\underset{l\neq j}{l=1}}^{\kappa} m_l C_{jl} +\lambda a_j=0, \quad j=1,\ldots,\kappa,
\label{eq:cc4}
\end{equation}
where
\begin{equation}
S_n =\frac14 \sum_{k=1}^{n -1} \frac1{\sin(k\pi/n )},
\label{eq:SN}
\end{equation}
and
\begin{equation}
C_{jl}=C_{jl}(a_j,a_l)=\sum_{k=1}^{n }
\frac{a_j-a_l\cos(\varpi_j-\varpi_l+2k\pi/n ) }{ (a_j^2+a_l^2-2a_ja_l\cos(\varpi_j-\varpi_l+2k\pi/n ) )^{3/2}},
\quad j,l=1,\ldots,{\kappa}.
\label{eq:coefs}
\end{equation}

Fixed all the polar angles of the leaders, System~(\ref{eq:cc4}) has $\kappa$ equations and $2\kappa +1$ unknowns: $\lambda$, the masses $m_j$ and the radii $a_j$, for $j=1,\ldots,\kappa$. Eliminating the term $\lambda$ subtracting the first equation from all the others, we obtain the following system of  $\kappa -1$ equations with $2\kappa$ unknowns,

\begin{equation}
(C_{j1}-S_n a_j)+\sum_{\underset{l\neq j}{l=2}}^{\kappa}(C_{jl}-a_jC_{1l})m_l+
\left(\dfrac{ S_n }{a_j^2}-a_jC_{1j}\right)m_j= 0, \; j=1,\ldots, \kappa.
\label{eq:cc5}
\end{equation}

Therefore, a twisted or nested $(\kappa,n)$--crown is obtained if and only if system~(\ref{eq:cc5}) is satisfied.

\begin{definition}
For any fixed value of $\kappa$ and $n$, and given a set of positive masses $\{m_1,m_2,\ldots,m_{\kappa}\}$ we say that $\{a_1,a_2,\ldots,a_{\kappa}\}$ is a set of \emph{admissible} values if there exists a $(\kappa,n)$--crown with masses $m_1,m_2,\ldots,m_{\kappa}$ satisfying system~(\ref{eq:cc5}). We denote by ${\cal A}_{\kappa}(n)$ the set of all admissible values.
\end{definition}

Notice that we can normalize system~(\ref{eq:cc5}) taken $m_1=1$ and $a_1=1$.

The case of nested $(2,n)$--crowns was solved by Moeckel and Sim\'o in \cite{Moeckel1995}. From now on we focus our attention to the twisted case.

\section{Twisted $(2,n)$-crowns}
\label{2tccnom0}

We consider the case of twisted $(2,n)$-crowns, so that $\varpi_1=0$ and $\varpi_2=\pi/n$.
Recall that the masses of one ring are equal to one and are located on a circle of radius one, whereas the
other ring with mass $m_2=m$ is located on a circle of radius $a_2=a$. Then, the set of Equations~(\ref{eq:cc5}) reduces to the following equation
\begin{equation}
(C_{21}(a)-S_{n}a)+\left(  \dfrac{S_{n}}{a^{2}}-aC_{12}(a)\right)
m=0,
\label{eq:cck2}
\end{equation}
where
\begin{equation}
\begin{array}{rcl}
C_{12}(a) =C_{12}(1,a)& = &\displaystyle\sum_{k=1}^{n}\frac{1-a\cos((2k-1)\pi/n)}%
{(1+a^{2}-2a\cos((2k-1)\pi/n))^{3/2}},\\
C_{21}(a) =C_{21}(a,1) & = &\displaystyle\sum_{k=1}^{n}\frac{a-\cos((2k-1)\pi/n)}%
{(1+a^{2}-2a\cos((2k-1)\pi/n))^{3/2}},
\end{array}
\label{eq:functionsA}
\end{equation}
are the functions defined in (\ref{eq:coefs}) for $\varpi_2-\varpi_1=\pi/n$.
Equation~(\ref{eq:cck2}) was also obtained by Roberts in \cite{PhDRoberts} and Yu and Zhang in \cite{YuZhang2012}.

As any central configuration problem, Equation~(\ref{eq:cck2}) is linear with respect to the mass $m$. Then, solving it as a function of the radius, $a>0$, we have that

\begin{equation}
m=H(a)=\dfrac{a^{2}(S_{n}a-C_{21}(a))}{S_{n}-a^{3}C_{12}(a)}=a^{2}\dfrac{F(a)}{G(a)}.
\label{eq:k2}%
\end{equation}

Clearly, any point of the half plane $(a,m)$ with $m>0$ satisfying (\ref{eq:k2}) (so that
 $a\in {\cal A}_2(n)$) correspond to a twisted $(2,n)$-crown.

Notice that $H(1)=1$, which means that if the second ring is located on the same circle than the first one, then the masses of all  bodies are equal, and the regular $2n$-gon is obtained.

Equation~(\ref{eq:cck2}) and (\ref{eq:k2}) are similar to the ones obtained by Moeckel and Sim\'o in \cite{Moeckel1995}
for the nested case $\varpi_1=\varpi_2$, changing $\cos((2k-1)\pi/n)$ by $\cos(2k\pi/n)$.
In the case of nested $(2,n)$--crowns, the authors prove that for any positive mass ratio $m$ there
exists only one central configuration, that is, only one admisible value $a>0$.
The proof in \cite{Moeckel1995}, is based in the fact that the function
$$ \sum_{k=1}^n \dfrac1{(1+a^{2}-2a\cos(k\pi/n))^{1/2}} $$
and all of its derivatives are positive. In our case the function involved is
$$ \phi(a)=\sum_{k=1}^n \dfrac1{(1+a^{2}-2a\cos((2k-1)\pi/n))^{1/2}}. $$
The function $F$, defined in~(\ref{eq:k2}) can be written as
$$ F(a)=S_na+\dfrac{d\phi}{da}(a).$$
In the case of the nested $(2,n)$--crowns, $F$ is the sum of two increasing functions in the interval $(0,1)$, which implies that there exists only one root of $F(a)=0$ in that interval. But for twisted crowns, $\phi$ and its derivatives do not have constant monotonicity, so we can not use the same arguments. Furthermore, function $H(a)$ behaves differently for $n=2,3,4$ or bigger than 5 and
we will see that, for each value of $m$, there can be more than one $(2,n)$-crowns, depending on the number of bodies in each ring. Thus, we have adopted a different approach.

The aim of this Section is twofold: first to describe the admissible set ${\cal A}_2(n)$, which is commonly called the \emph{inverse problem}, and second to study the number of central configurations that exist for any fixed value of $m>0$. Both objectives are obtained through the study of the function $H(a)$.

We start with some useful properties of the functions $H$, $F$ and $G$.
\begin{lemma}
\label{lemma1}
Let $H(a)$, $F(a)$ and $G(a)$ be the functions defined in (\ref{eq:k2}) for
$a\in[0,\infty)$.  Then, $F(a)$ and $G(a)$ are analytic functions, and for any $a>0$ the following statements hold:
\begin{enumerate}

\item $F(a)=a\,G(1/a)$,

\item $H(1/a)=1/H(a)$.

\end{enumerate}
\end{lemma}
The proof is straightforward from the definitions and the fact that $C_{12}(1/a)=a^{2}C_{21}(a)$ for $a>0$.

Two important consequences follow from these two properties.
From the first one we get that, to know the sign of the function $H$, it is enough to know the sign
of $F$ (or $G$).
And from the second property we have that if $a\in{\cal A}_{2}(n)$, so
it is $1/a$. That means that the $(2,n)$-crowns corresponding to $(a,m)$ and $(1/a,1/m)$ are
qualitatively the same, in the sense that one is just the other one conveniently scaled.
In this paper, two central configurations are considered ``different'' if one cannot be obtained from the other one using this symmetry. Due that, we only consider values of $m\geq 1$.

The study reveals that different results are obtained for $n=2$, $n=3$ or $n\geq 4$. The case $n=2$ is already known till 1932 from paper by MacMillan and Bartky, \cite{MacMillanBartky} (see also \cite{Zhang2002}): for any positive value of $m$, there exists only one admissible value of the radius $a\in (1/\sqrt{3},\sqrt{3})$ such that the two twisted segments are a central configuration of the 4-body problem. Furthermore, it is not difficult to see that $m>1$ for $a<1$, so that the bigger masses are always located in the inner ring, and the limit values  $a=\sqrt{3}$ and $a=1/\sqrt{3}$ correspond to, respectively, the limit cases $m=0$ and $m=\infty$.

Next, we study the cases $n=3$ and $n\geq 4$ separately.

\subsection{Twisted $(2,3)$--crowns}\label{sect-n3}

We consider two twisted equilateral triangles. First, we will see that the set of all admissible values is the union of three disjoin intervals.
Second, we will show that given $m>0$, the number of central configurations varies from one to three.

The rang of admissible values depends on the number of zeros of $F(a)$ and their location.
\begin{prop}
\label{propF3}
Let $F$ be the function defined in (\ref{eq:k2}). Then, for $n=3$, the following statements hold.
\begin{enumerate}
  \item The equation $F(a)=0$ has exactly two positive solutions, denoted by $z_1$ and $z_2$;
  \item $3/8<z_1<7/16$ and $25/16<z_2<13/8$;
  \item $F(a)<0$ for $a\in (z_1,z_2)$ and $F(a)>0$ for $a\in(0,z_1)\cup (z_2,\infty)$.
\end{enumerate}
\end{prop}
The details of the proof are given in the Appendix.

\begin{theo}
\label{theoA3}
The set of admissible values for a $(2,3)$-crown is
$$ {\cal A}_2(3)= (0,z_1) \cup (1/z_2,z_2) \cup (1/z_1,\infty),$$
where $z_1$ and $z_2$ are given in Proposition~\ref{propF3}. Furthermore, for any positive value of $m$ there
exist an admissible value $a=H^{-1}(m)$.
\end{theo}
\begin{proof}
Using Lemma~\ref{lemma1} the two positive solutions of $G(a)=0$ are $1/z_2$ and $1/z_1$. From Proposition \ref{propF3},
$z_1 z_2<1$ and therefore $0<z_1<1/z_2<z_2<1/z_1$.  Furthermore, $G(a)<0$ only for $a\in(1/z_2,1/z_1)$.
Combining the signs of the two functions $F$ and $G$, the admissible values $a>0$ are
$$(0,z_1) \cup (1/z_2,z_2) \cup (1/z_1,\infty).$$
Using the fact that $\lim_{a\to 1/z_2^+}  H(a)=\infty$ and $H(z_2)=0$  we have that for any positive value of $m$ there exists a value of the radius $a>0$ within the above admissible set ${\cal A}_2(3)$.
\end{proof}

As in the case $n=2$, there exist values of $a>0$ which are not admissible. In this case, there are two annuli in which the second ring can not be placed. The main difference with the case $n=2$, is that the second ring can be placed as far as we want, or as close to the origin as we want. See the approximate values of $z_1$ and $z_2$ in Table~\ref{table1}.

Notice that $H(0)=H(z_1)=0$, and the function $H$ is continuous in $a\in(0,z_1)$,  so that there exists a maximum value of $H$ in this interval. Therefore, only masses up to a certain value can be placed in the second ring with radius
$a\in (0,z_1)$. Due to the symmetry given by Lemma \ref{lemma1}, the opposite occurs in the interval $(1/z_1,\infty)$, where the function $H$ tends to infinity at both limits of the interval, so that it has a minimum. Thus, only masses bigger than a certain value can be placed far from the origin. See Figure~\ref{fig:FunH_n3}.
\begin{figure} [!h]
    \centering
    \includegraphics{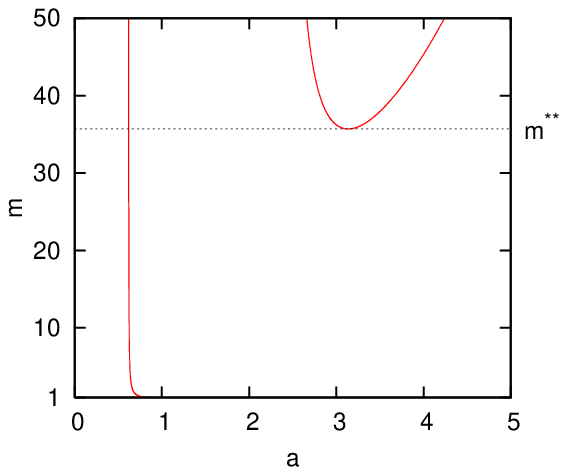}
    \includegraphics{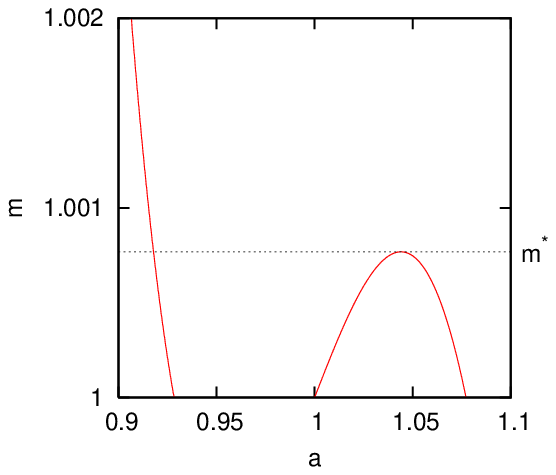}
\caption{Graph of the function $H$ for $n=3$ and $m>1$ (right: detail for $a\in[0.9,1.1]$). The horizontal dotted lines correspond to
the values $m=m^{**}$ and $m=m^*$ (see Theorem~\ref{theoCC3}).}
\label{fig:FunH_n3}
\end{figure}

From these observations, we have that the number of central configurations that exist for a fixed mass ratio $m>0$ can be bigger than one. Next theorem states exactly the number of central configurations for each positive mass ratio.

\begin{theo} \label{theoCC3}
For any positive value of the mass ratio $m\geq 1$, there exists at least one central configuration of the 6-body problem corresponding to a twisted $(2,3)$--crown. Furthermore, there exist values $m^*,m^{**}>1$ such that
\begin{enumerate}
	\item for any $m\in(1,m^*)$, there exists exactly three different twisted $(2,3)$--crowns, all of them
	with admissible values $a\in (1/z_2,z_2)$;
	\item for any $m\in(m^{**},\infty)$, there exists exactly three different twisted $(2,3)$--crowns, one of them
	with admissible value $a\in (1/z_2,1)$ and the other two with admissible values $a\in (1/z_1,\infty)$;
	\item for any $m\in (m^*,m^{**})$ there exists only one twisted $(2,3)$--crown with admissible value $a\in (1/z_2,1)$,
	\item for $m=1,m^*,m^{**}$ the number of different twisted $(2,3)$--crowns is exactly two;
\end{enumerate}
where $z_1$ and $z_2$ are given in Proposition~\ref{propF3}.
\end{theo}
\begin{proof}
From Proposition~\ref{theoA3}, we have that for any positive value of $m$ there exist at least one central configuration.
To establish the number of central configurations, we look for the number of positive solutions of the equation $H(a)=m$, for a fixed positive value $m\geq 1$. In order to do that, we study the monotonicity and the critical points of the function $H$ in the admissible set
$$[1,z_2) \cup (1/z_1,\infty).$$
Using Lemma~\ref{lemma1}, we recover the behavior of $H(a)$ in $(0,z_1]\cup(1/z_2,1]$. 
On one hand, the function $H$ has only two critical points for $a>1$, at $a^*\in(1,z_2)$ and $a^{**}\in(1/z_1,+\infty)$ (see 
Appendix, Lemma~\ref{lemH3}).
Using that  $H(1)=1$, $H(z_2)=0$, and $\lim_{a\to (1/z_1)^+} H(a)$ and
$$\lim_{a\to +\infty} H(a)=\lim_{a\to 0}\dfrac1{H(a)}=+\infty,$$
it follows that
$a^*$ is a local maximum and $a^{**}$ is a local minimum. On the other hand, 
the equation $H(a)=a$ has only one positive solution at $a=1$ (see again Lemma~\ref{lemH3}). This implies that $H(a^*) < H(a^{**})$.
Defining $m^*=H(a^*) < m^{**}=H(a^{**})$, we conclude the proof (see Figure~\ref{fig:FunH_n3}).
\end{proof}

The approximate bifurcation values are $m^*=1.0007682$ and $m^{**}=35.70017694$.

\subsection{Twisted $(2,n)$--crowns with $n\geq 4$} \label{sect-n4}
Next, we consider the case of $n\geq 4$. We will see that, as in the case of $(2,3)$--crowns, the set of all admissible values consists in three disjoint intervals, and there are two annuli where no ring can be placed. On the contrary, when $n\geq 4$, for any value of the mass parameter $m>0$, there always exist at least three central configurations.

Following the arguments of the previous subsection, next proposition states some properties about the zeros of the function $F$.

\begin{prop}
\label{propFn}
Let $F$ be the function defined in (\ref{eq:k2}). Then, for $n\geq 4$, the following statements hold.
\begin{enumerate}
  \item The equation $F(a)=0$ has at least two positive solutions, denoted by $z_1$ and $z_2$.
  \item Suppose that $z_1$ and $z_2$ are the only positive solutions of $F(a)=0$. Then
  \begin{enumerate}
    \item $z_1 < 1 < z_2$ and $z_1z_2 >1$.
    \item $F(a)<0$ for $a\in(z_1,z_2)$, and $F(a)>0$ for $a\in(0,z_1)\cup(z_2,\infty)$.
  \end{enumerate}
\end{enumerate}
\end{prop}
The details of the proof can be found in the Appendix.

At this point we would like to remark that we have not been able to proof analytically that the number of positive solutions of $F(a)=0$ is exactly two, as numerical simulations suggest. The inability to proof the uniqueness of the two zeros of the function $F$ can be found also in Roberts in \cite{PhDRoberts}, where he studied the existence of central configurations of two twisted rings plus a central mass. The same situation occurs in \cite{2004Bangetal}, where the relative equilibria positions for an infinitesimal mass submitted to the Newtonian attraction of $n$ bodies in a regular $n$-gon are studied. As fas as we know, the uniqueness of the two zeros of the function $F$ has not been solved.

From now on we will assume the following conjecture.

\begin{conj}\label{conj1}
The function $F$, defined in (\ref{eq:k2}), has only two positive solutions, denoted by $z_1$ and $z_2$.
\end{conj}

In Table~\ref{table1} we show the (approximate) values of $z_1$ and $z_2$ for some values of $n$.
\begin{table}[!ht]
\centering
\begin{tabular}{|c|c|c|c|}
\hline
$n$ & $z_1$ & $z_2$
\\
\hline
    3  &  0.413887932417 & 1.619789608802 \\
\hline
    4  &     0.697380509876    &    1.602408486212  \\
\hline
    5    &         0.822182869908 &       1.597921728909
\\
\hline
    6     &        0.884321138125    &    1.592235355387
\\
\hline
    7      &       0.918990363772       & 1.584120901279
\\
\hline
    8       &      0.940138179122  &      1.574515176634
\\
\hline
    9     &        0.953949939513   &     1.564321826382
\\
\hline
   10   &          0.963459881269   &     1.554123467683
   \\
\hline
   $\vdots$ &  $\vdots$ &  $\vdots$
\\
\hline
100 & 0.999674025507    &   1.352557858581
\\
\hline
500 & 0.999986989988   &    1.279569044474
\\
\hline
1000 & 0.999996754292  &     1.256683821749
\\
\hline
5000 & 0.999999869916   &    1.215703126473
\\
\hline
\end{tabular}
\caption{Values of $z_1$ and $z_2$ for the values of $n$ shown.}
\label{table1}
\end{table}

Next result states the rang of admissible values of a twisted $(2,n)$--crown.
\begin{theo}
\label{theoAn}
For $n\geq 4$, the set of admissible values is
\[
{\cal A}_2(n) =(0,1/z_2) \cup(z_1,1/z_1) \cup(z_2,\infty).
\]
where $z_1$ and $z_2$ are given in Conjecture~\ref{conj1}.
\end{theo}

\begin{proof}
Suppose that $z_1$ and $z_2$ are the only roots of $F(a)=0$. Using Lemma~\ref{lemma1}
and Proposition~\ref{propFn}, we have that $1/z_1$ and
$1/z_2$ are the only positive solutions of $G(a)=0$,
$$ 1/z_2 < z_1 < 1 < 1/z_1 < z_2.$$
and $H(a)>0$ only if $(0,1/z_2) \cup(z_1,1/z_1) \cup(z_2,\infty)$.
\end{proof}

Next theorem gives the minimum number of central configurations for any given value of the mass ratio $m\geq 1$.

\begin{theo}
\label{theoCCn}
Let be $n\geq 4$ and $z_1$, $z_2$ given in Conjecture~\ref{conj1}. Then, there exists an admissible value $a_1>1$ such that
\begin{enumerate}
  \item for any positive value of the mass ratio $m> 1$ there exists at least three different twisted  $(2,n)$--crowns,
one with admissible value $a\in(1/a_1,1/z_2)$, one with $a\in(1,1/z_1)$ and one with $a\in(a_1,\infty)$;
  \item for $m=1$ there exist at least two different twisted  $(2,n)$--crowns. If $n$ is even, there are exactly two, with admissible values $a=1$ and $a=a_1$, respectively.
  \end{enumerate}
\end{theo}

\begin{proof}
From Theorem~\ref{theoAn} we know that the mass ratio $m=H(a)$ is positive when
$a\in (0,1/z_2) \cup(z_1,1/z_1) \cup(z_2,\infty)$. Recall that $H(0)=H(z_1)=H(z_2)=0$ and that
$\lim_{a\to (1/z_i)^-}H(a)=+\infty$, $i=1,2$. Furthermore
$$\lim_{a\to+\infty} H(a)=\lim_{a\to+\infty}a^2 \dfrac{F(a)}{aF(1/a)}=+\infty$$
using Lemma~\ref{lemma1} and Lemma~\ref{lemFn} (see Appendix).
Thus, any horizontal line in the half plane $(a,m)$, $m>0$ intersects at least three times the curve
$m=H(a)$. Due to the symmetry $H(1/a)=1/H(a)$ (see Lemma~\ref{lemma1}), the three intersections
correspond to three different central configurations for $m> 1$.
To conclude, the proof of the first statement just take $a_1$ a solution greater than one of $H(a)=1$.

For $m=1$ the previous arguments give that there exists at least two different central configurations.
Let us consider from now on that $n$ is even. We want to show that the equation $H(a)=1$ has exactly
 three solutions at $a_1$, $1$ and $1/a_1$. Due to the symmetry of the problem, 
the admissible values $a_1$ and $1/a_1$ leads to the same central configuration. Then, there exists exactly
two $(2,n)$-crowns when $m=1$ and $n$ is even. 

The equation $H(a)=1$ can be written as
$$ S_n \dfrac{a^3-1}{a^2}=(a^2-1)\sum_{k=1}^n \dfrac{\cos \theta_k}{d_k^{3/2}},$$
where $\theta_k=(2k-1)\pi/n$ and $d_k=1+a^2-2a\cos\theta_k$, $k=1,\ldots,n$.
Clearly $a=1$ is a solution (that was already known), so considering $a\neq 1$ we have to solve the equation

$$g_1(a):=S_n \dfrac{a^2+a+1}{a^2(a+1)} = \sum_{k=1}^n \dfrac{\cos \theta_k}{d_k^{3/2}} =: g_2(a).$$

From Lemma~\ref{lemteoCCn} (see Appendix) we have the following properties:
\begin{enumerate}
\item $g_i(1/a)=a^3g_i(a)$, $i=1,2$;
\item $\exists$ $A>0$ such that for all $a>A$, $g_2(a) < g_1(a)$;
\item when $n$ is even, $g_2(1) > g_1(1)$;
\item  when $n$ is even, $g_1$ and $g_2$ are decreasing functions in $(1,\infty)$.
\end{enumerate}

From the first property we have that if $a_1$ is a solution of $g_1(a)=g_2(a)$, so it is $1/a_1$. Thus we only look for solutions in the interval $(1,\infty)$.
Using the second and third properties and applying Bolzano's theorem, there exists at least one solution.
Finally, from the last property, the solution must be unique.
\end{proof}

In other to obtain an exact number of twisted $(2,n)$--crowns we we should be able to prove that the function $H$ is strictly increasing.
We have not succeeded in that. All the numerical experiments done support the conclusion that the number is exactly three. See Figure~\ref{fig:plot_Hns}, where the set of curves $m=H(a)$ is plotted for three different values of $n$.

\begin{figure} [!h]
\centering
\includegraphics{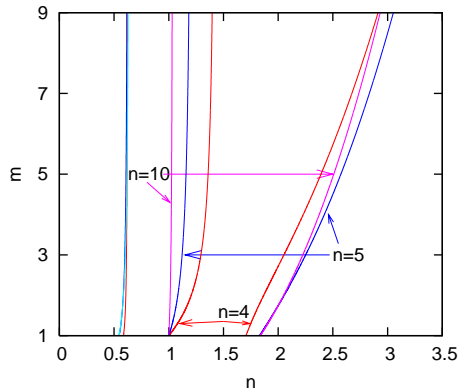}
\caption{The curves $m=H(a)$ for $n=4, 5, 10$.}
\label{fig:plot_Hns}
\end{figure}

\section{Convex twisted $(2,n)$--crowns} \label{convexCC}

We will discuss in this section whether the central configurations obtained are convex or concave.
It is not difficult to see that, given $2n$ bodies in two $n$--gons of radius $1$ and $a$, respectively,
the configuration is convex if and only if
\begin{equation}
a\in {\cal C}_n = \left(\cos \left( \dfrac{\pi}n\right), \dfrac1{\cos \left( \dfrac{\pi}n\right) } \right).
\label{eq:convex}
\end{equation}

For any $n\geq 3$ we have seen that the set of admissible values consists in three disjoint intervals. From now on we will focus our interest on
the interval that contains the value $a=1$, that we call the \emph{central admissible interval}, ${\cal A}_{2C}(n)$. Our purpose is to
prove the following statement.

\begin{prop} \label{prop-ccc}
Consider twisted $(2,n)$--crowns of radius 1 and $a$, with masses $1$ and $m$ respectively. Let $z_1$ and $z_2$ given in Proposition~\ref{propF3} (for $n=3$) and Proposition~\ref{propFn} and Conjecture~\ref{conj1} (for $n\geq 4$). 
\begin{enumerate}
  \item for $n=3$,  $ {\cal A}_2(3)\cap {\cal C}_3={\cal A}_{2C}(3)$, that is, a twisted $(2,3)$--crown is convex if and only if $a$ belongs to the central admissible interval, ${\cal A}_{2C}(3)$.
  \item for $n=4$, $ {\cal A}_2(4)\cap {\cal C}_4={\cal C}_4$, that is, if the twisted $(2,4)$--crown is convex then $a$ belongs to the central admissible interval, ${\cal A}_{2C}(4)$, but not all the values of $a$ within ${\cal A}_{2C}(4)$ correspond to convex configurations.
  \item For $n\geq 5$, if the twisted $(2,n)$--crown is convex, then $a$ belongs to the central admissible interval.
\end{enumerate}
\end{prop}
\begin{proof}
For $n=3$, by Proposition~\ref{propF3} we have that
$$ z_1 < \dfrac7{16} < \cos\left(\dfrac{\pi}3\right) = \dfrac12 < \dfrac8{13} < \dfrac1{z_2}.$$
Therefore, using Theorem~\ref{theoA3},
$$ {\cal C}_3 \cap {\cal A}_2(3) = \left( \dfrac1{z_2},z_2\right)={\cal A}_{2C}(3).$$

For $n=4$, and by Proposition~\ref{propFn}, $F(a)<0$ for $a\in (z_1,z_2)$, and $F(a)>0$ otherwise. Using the expression of
$F(a)$ for $n=4$ (see the proof of Proposition~\ref{propFn} in the Appendix) we have that
$$F\left(\cos\left(\dfrac{\pi}4\right)\right)= \dfrac{\sqrt{2}}8+\dfrac12-\dfrac8{5\sqrt{5}} <0.$$
Therefore, $z_1 < \cos(\pi/4)$ and
$$ {\cal C}_4 \cap {\cal A}_2(4) = {\cal C}_4 \subset \left(z_1,\dfrac1{z_1}\right)={\cal A}_{2C}(4).$$

Finally for $n\geq 5$, by Theorem~\ref{theoAn} the central admissible interval is $(z_1,1/z_1)$ and
by Lemma~\ref{lemZn} we have that $\frac1{z_2} < 1-\frac1n < z_1$. Using the fact that $1-\frac1n < \cos(\pi/n)$
we obtain that
$$ {\cal C}_n \cap \left(  (0,\frac1{z_2})\cup (z_2\infty) \right) =\emptyset.$$
Therefore,
$$ {\cal C}_n \cap {\cal A}_n =
 \left\{ \begin{array}{c}
 {\cal C}_n \subset (z_1,\frac1{z_1}) \\
 \hbox{or}\\
 (z_1,\frac1{z_1})
 \end{array}\right.
$$
\end{proof}
For $n\geq 5$ we have not been able to proof that all the convex central configurations correspond exactly with the
central admissible interval, although the numerical explorations point in that direction.
In Table~\ref{table2} we show the values of central admissible intervals for some values of $n$.
In Figure~\ref{fig:ccc}
we show, for $n\in (5, 5000)$ the difference $\Delta_n=z_1-\cos(\pi/n)$ in logarithmic scale. We can see that
the difference is always positive, which implies that ${\cal C}_n \cap {\cal A}_2(n) =  (z_1,\frac1{z_1})={\cal A}_{2C}(n)$.

\begin{table}[!ht]
\centering
\begin{tabular}{|c|c|c|}
\hline
$n$ & Central admissible interval ${\cal A}_{2C}(n)$ & $\cos(\pi/n) $
\\
\hline
    3  &  (0.617364128382, 1.619789608802)  &  0.5 \\
\hline
    4  &    (0.697380509876,  1.433937406966) &  0.707106781187\\
\hline
    5    &  (0.822182869908, 1.216274428233) &    0.809016994375
\\
\hline
    6     &  (0.884321138125, 1.130810920250)   &  0.866025403784
\\
\hline
    7      &  (0.918990363772, 1.088150691695)      & 0.900968867902
\\
\hline
    8       &  (0.940138179122, 1.063673428234)  & 0.923879532511
\\
\hline
    9     &    (0.953949939513, 1.048273036749)   &   0.939692620786
\\
\hline
   10   &      (0.963459881269, 1.037925936971)  &   0.951056516295
   \\
   \hline
   $\vdots$ &  $\vdots$ &  $\vdots$
\\
\hline
100 & (0.999674025507, 1.00032608079) & 0.999506560366
   \\
   \hline
500 & (0.999986989988, 1.00001301018) & 0.999980339576
   \\
   \hline
1000 & (0.999996754292, 1.00000324572) & 0.999995075057
   \\
   \hline
5000 & (0.999999869916, 1.00000013008) & 0.999999802608
\\
\hline
\end{tabular}
\caption{The central admissible interval and $\cos(\pi/n)$ for different the values of $n$.}
\label{table2}
\end{table}

\begin{figure}[!ht]
\centering
\includegraphics{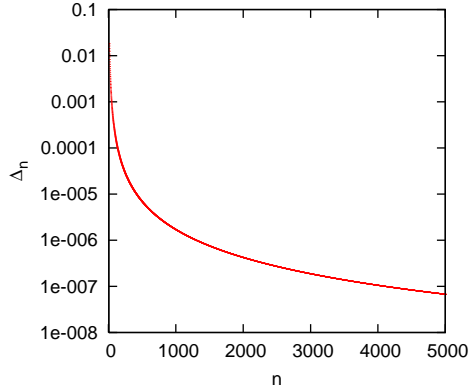}
\caption{For $5\geq n \geq 5000$, the difference $\Delta_n=z_1-\cos(\pi/n)$ (in logarithmic scale).}
\label{fig:ccc}
\end{figure}

Combining the results of Proposition~\ref{prop-ccc} and Theorems~\ref{theoCC3} and \ref{theoCCn}, we have the following
summary on convex central configurations.
\begin{enumerate}
\item For $n=3$, let $m^*$ be the bifurcation value given by Theorem~\ref{theoCC3}. Then, for any positive value $m$ there exists at least one convex twisted $(2,3)$--crown and
\begin{enumerate}
	\item for any $m\in(1,m^*)$, there exist exactly three different convex twisted $(2,3)$--crowns,
	\item for any $m\in (m^*,\infty)$ there exists only one convex twisted $(2,3)$--crown,
	\item for $m=1,m^*$ the number of different convex twisted $(2,3)$--crowns is exactly two.
\end{enumerate}
When $m=1$, the two convex twisted $(2,3)$--crowns where already numerically computed by Moeckel in an unpublished report.

\item For $n=4$, there exists a value for the mass ratio $m=\overline{m}$ such that
\begin{enumerate}
	\item for any $m\in[1,\overline{m}]$, there exists exactly one convex twisted $(2,4)$--crown,
	\item for any $m\in (\overline{m},\infty)$ there exists no convex twisted $(2,4)$--crown.
\end{enumerate}
The approximate value of $\overline{m}$ is $H(1/\cos(\pi/4))=16.05679941$.

\item For $n\geq 5$, if $H(1/\cos(\pi/n))>0$, then for $m \leq H(1/\cos(\pi/n))$, there exists at least one convex twisted $(2,n)$--crown.
On the contrary, if $H(1/\cos(\pi/n))<0$, then for all values of $m$ there exists at least one convex twisted $(2,n)$--crown.
\end{enumerate}

\section{Appendix}

We include here a detailed proof of some of the results presented in the paper. We have listed first all the lemmas, and then 
the propositions. In some cases, one can convince himself
of the certainty of the results simply by plotting the graphic of a function. Nevertheless,
we give an analytical proof of each one of them.


\begin{lemma}
\label{lemH3}
Let $H$ be the function given in (\ref{eq:k2}). For $n=3$ the following statements hold.
\begin{enumerate}
	\item the function $H$ has only two critical points for $a>1$, $a^*\in(1,z_2)$ and
$a^{**}\in(1/z_1,+\infty)$, where $z_1$ and $z_2$ are given in Proposition~\ref{propF3},
	\item the equation $H(a)=a$ has only one positive solution at $a=1$.
\end{enumerate}
\end{lemma}
\begin{proof}
Recall that $H(a)=a^2 \frac{F(a)}{G(a)}$, and for $n=3$, $F$ is given by (\ref{eq:funFn3}) and
$$
G(a) = \dfrac{\sqrt{3}}3  - a^3 \left( \dfrac{2-a}{(1+a^2-a)^{3/2}}+\dfrac1{(1+a)^2} \right).
$$

We start proving the first statement. The numerator of $H'(a)$ can be written as
$$ -3\sqrt{3}a\sqrt{1+a^2-a} \left(2\sqrt{1+a^2-a} \;q(a)-(1+a)p(a)\right),$$
where $p$ and $q$ are the following polynomials of degree 10:
\begin{eqnarray*}
p(a) &=& 4\,{a}^{10}-{a}^{9}-23\,{a}^{8}- \left(2\,\sqrt {3}+23 \right) {a}^{7}
- \left(17\,\sqrt {3}+1 \right) {a}^{6}\\
& & + \left( 42\,\sqrt {3}+8 \right) {a}^{5}- \left(17\,\sqrt {3}+1 \right) {a}^{4}
- \left(2\,\sqrt {3}+23 \right) {a}^{3}-23\,{a}^{2}-a+4,
\\
q(a) &=& 2\,{a}^{10}- \left(2+\sqrt {3} \right) {a}^{9}- \left(2\,\sqrt {3}-2 \right) {a}^{8}
- \left(2\,\sqrt {3}-2 \right) {a}^{7}- \left(12\,\sqrt {3}+2 \right) {a}^{6} \\
& &- \left(13\,\sqrt {3}-4 \right) {a}^{5}- \left(12\,\sqrt {3}+2 \right) {a}^{4}
- \left( 2\,\sqrt {3}-2 \right) {a}^{3}- \left(2\,\sqrt {3}-2 \right) {a}^{2}
\\
& &- \left( 2+\sqrt {3} \right) a +2.
\end{eqnarray*}

The equation $H'(a)=0$ is equivalent to solve
$$g_1(a)= \dfrac{2\sqrt{1+a^2-a}}{1+a} = \dfrac{p(a)}{q(a)} = g_2(a).$$
Recall that we look for the solutions of $g_1(a)=g_2(a)$ in the interval $[1,\infty)$.
On one hand, the function $g_1\in{\cal C}^1$, is positive, $g_1(1)=1$, increases in $[1,\infty)$, and
$\lim_{a\to\infty}g_1(a)=2$.
On the other hand, $g_2(1)={\frac {76+3728\sqrt {3}}{6563}}\, <1$ and $\lim_{a\to\infty}g_2(a)=2$. Furthermore,
we will prove that the function $g_2$, for $a\geq 1$, has only one pole $u\in(46/16,47/16)$, a zero $v\in (47/16,3)$,
only one critical point $w\in (5,6)$ which is a local maximum, and the function increases in $(1,u)\cup(u,w)$ and decreases $(w,+\infty)$.
From all of the above properties of the functions $g_1$ and $g_2$, there exist only two values $a^*\in(1,u)$ and $a^{**}\in(v,w)$ at which $H'(a)=0$. See Figure~\ref{fig:prooftheo1}.
Finally, using a Bolzano's argument and the bounds of $z_1$ and $z_2$ given in Proposition~\ref{propF3}
$a^* < 25/16<z_2$ and $a^{**} > 3 > 1/z_1$.
\begin{figure}[!ht]
\centering
\includegraphics[scale=0.3]{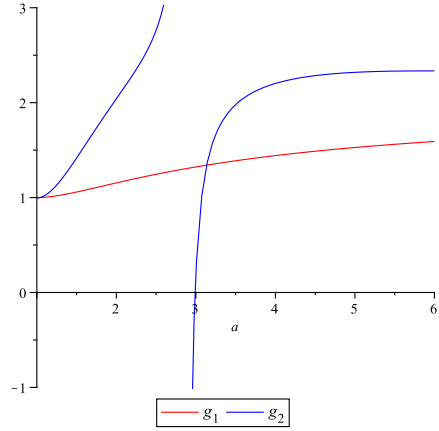}
\label{fig:prooftheo1}
\caption{Functions $g_1(a)$ and $g_2(a)$ used in the proof of Lemma~\ref{lemH3}}
\end{figure}

We have to prove the behavior of the function $g_2$ stated above.
Notice that the polynomials satisfy that $a^{10}p(1/a)= p(a)$ and $a^{10}q(1/a)= q(a)$. That means that if $a$ is a zero of one of the polynomials, so it is $1/a$. Applying this property and Descarte's rule, $q(a)$ has only one zero at $a=u\geq 1$, and, by Bolzano's theorem, $u\in(46/16,47/16)$.
To determine the number of zeros of $p(a)$, we consider the polynomial $p(1+a)$. Applying again Descarte's rule, $p(1+a)$ has only one positive zero, and therefore, $p(a)$ has only one zero at $a=v\geq 1$, which is located in
$(47/16,3)$. Looking at the sign of the polynomials $p(a)$ and $q(a)$, $g_2(a)>0$ for $a\in[1,u)\cup (v,+\infty)$ and $g_2(a)<0$ for $a\in(u,v)$.

It remains to see that $g_2$ has only one critical point in $(1,\infty)$, $a=w>u$. We need to solve
$$p'(a)q(a)-p(a)q'(a) = -(3+2\sqrt{3})(a+1)(a-1)r(a)=0,$$
where $r(a)$ is a polynomial of degree 16 that can be written as
$$ r(a) = r_1(a) -r_2(a),$$
\begin{eqnarray*}
r_1(a) &=& 2(a^{16}+1) +(220-106\sqrt{3})a^3(a^{10}+1)+(1507-744\sqrt{3})a^4(a^8+1) \\
r_2(a) &=& (88\sqrt{3}-140)a(a^{14}+1)+(109\sqrt{3}-182)a^2(a^{12}+1)+(1276-574\sqrt{3})a^5(a^6+1)
\\
& & + (264\sqrt{3}-356)a^6(a^4+1)+(3868-1814\sqrt{3})a^7(a^2+1)+(560\sqrt{3}-351)a^8
\end{eqnarray*}
By Descarte's rule, $r(a+2)$ has only one positive solution, so that $r(a)$ has only one zero greater than 2,
which is located $a=w\in(5,6)$.  The polynomials $r_1$ and $r_2$ are positive and increasing for $a>0$, and
$$ r_1(1) < r_2(1), \qquad r_1(2) < r_2(2),$$
Therefore, $r(a) <0$, for $a\in(1,2)$, so that it has no zeros in that interval, and $w$ is the only critical point of $g_2$ for $a>1$. This concludes the proof of the first item.

The second statement claims that the only positive solution of $H(a)=a$ is $a=1$. The equation can be written as
$aF(a)=G(a)$, and after a simplification, it can be written as
$$ \dfrac{\sqrt{3}}3 (a^2-1) = \dfrac{a(a-1)(a^2-a+1)}{(a^2-a+1)^{3/2}}+\dfrac{a(1-a^2)}{(1+a)^2}.$$
Clearly the equation has $a=1$ as a solution. Simplifying by the term $a-1$ at both sides we obtain
$$ \dfrac{\sqrt{3}}3 (a+1) = \dfrac{a}{(a^2-a+1)^{1/2}}-\dfrac{a}{(1+a)},$$
or equivalently
$$ \dfrac{\sqrt{3}}3 (a+1) + \dfrac{a}{(1+a)} = \dfrac{a}{(a^2-a+1)^{1/2}}.$$
The study of the derivative of the functions of both sides of the equations leads quickly to the conclusion that the equation has no positive solutions. Therefore, $a=1$ is the only positive solution of $H(a)=0$.
This concludes the proof.
\end{proof}

\begin{lemma}
\label{lemFn}
Let $F$ be the function defined in (\ref{eq:k2}) for $n\geq 4$. Then
\begin{enumerate}
	\item $F(1)=S_n-C_{21}(1)<0$;
	\item $\displaystyle \lim_{a\to 0^+} \frac{F(a)}{a} = S_n+\frac{n}2$;
    \item $\displaystyle \lim_{a\to+\infty}F(a) = +\infty.$
\end{enumerate}
\end{lemma}
\begin{proof}
The first two results correspond to Lemma 1.2.4 of the first Chapter of \cite{PhDRoberts}. 
Notice that the notation used by the author is different from the one used here. 
The last result is obtained from Lemma~\ref{lemma1} and the fact that $G(0)=S_n>0$.
\end{proof}

\begin{lemma}
\label{lemZn}
Let be $n\geq 5$ and suppose that $z_1$ and $z_2$ are the only positive solutions of $F(a)=0$. Then
$$0 < 1-\frac1{n} < z_1< 1 < \dfrac1{z_1} < \frac{n}{n-1} < z_2. $$
\end{lemma}
\begin{proof}
  This result correspond to Lemma 1.2.5 of the first Chapter of \cite{PhDRoberts}.
\end{proof}

\begin{lemma}
\label{lemteoCCn}
Let be $n\geq 4$  and
$$g_1(a)=S_n \dfrac{a^2+a+1}{a^2(a+1)}, \qquad
g_2(a)= \sum_{k=1}^n \dfrac{\cos \theta_k}{d_k^{3/2}},$$
where $d_k=1+a^2-2a\cos\theta_k$, and $\theta_k=(2k-1)\pi/n$ for $k=1,\ldots,n$.
They satisfy the following properties:
\begin{enumerate}
\item $g_i(1/a)=a^3g_i(a)$, $i=1,2$;
\item $\exists$ $A>1$ such that for all $a>A$, $g_2(a) < g_1(a)$;
\item if $n$ is even, then $g_2(1) > g_1(1)$;
\item if $n$ is even, then $g_1$ and $g_2$ are decreasing functions in $(1,\infty)$.
\end{enumerate}
\end{lemma}
\begin{proof}
  The first item is a straightforward computation. The second one comes from the limit
  $$\lim_{a\to\infty}\dfrac{g_2(a)}{g_1(a)} = \lim_{\varepsilon \to 0} \dfrac{g_2(1/\varepsilon)}{g_1(1/\varepsilon)}
  = \lim_{\varepsilon \to 0} \dfrac{\varepsilon^2(1+\varepsilon) g_2(\varepsilon)}{S_n(\varepsilon^2+\varepsilon+1)} =0.$$

Now let us consider $n=2m$ with $m\geq 2$.
Then
$$ g_2(a)=2\sum_{k=1}^{[m/2]} \dfrac{\cos \theta_k}{(1+a^2-2a\cos\theta_k)^{3/2}}-
\dfrac{\cos \theta_k}{(1+a^2+2a\cos\theta_k)^{3/2}},$$
where $[\cdot]$ denotes the floor function, and
$$ g_2'(a)=-6\sum_{k=1}^{[m/2]} \cos \theta_k\left(\dfrac{a-\cos \theta_k}{(1+a^2-2a\cos\theta_k)^{5/2}}-
\dfrac{a+\cos \theta_k}{(1+a^2+2a\cos\theta_k)^{5/2}}\right).$$
Notice that $\cos\theta_k >0$ for $k=1,\ldots,[m/2]$. Thus, to prove the forth statement we only need to see that
$$ \dfrac{a-\cos \theta_k}{(1+a^2-2a\cos\theta_k)^{5/2}}-
\dfrac{a+\cos \theta_k}{(1+a^2+2a\cos\theta_k)^{5/2}} >0,$$
which is equivalent to see that
$$ (a-\cos \theta_k)^2(1+a^2+2a\cos\theta_k)^5 - (a+\cos \theta_k)^2(1+a^2-2a\cos\theta_k)^5 >0.$$
The property follows from the fact that the polynomial
$$q=(a-b)^2(1+a^2+2ab)^5-(a+b)^2(1+a^2-2ab)^5 >0$$
for all $a>1$ and $0<b\leq 1$.

Finally,
\begin{eqnarray*}
  g_1(1) &=& \dfrac38 \sum_{k=1}^{n-1} \dfrac1{\sin(k\pi/n)} =
  \dfrac38 \left( 1 + 2 \sum_{k=1}^{m-1} \dfrac1{\sin(k\pi/n)}
  \right)
   \\
   &<&  \dfrac38 \left( 1+ \dfrac{n-2}{\sin(\pi/n)} \right)
\\
  g_2(1) &=& \dfrac28 \sum_{k=1}^m \dfrac{\cos\theta_k}{\sin^3(\theta_k/2)} =
  \dfrac28 \left(
  \sum_{k=1}^{[m/2]} \dfrac{\cos\theta_k}{\sin^3(\theta_k/2)}
  + \sum_{k=[m/2]+1}^m  \dfrac{\cos\theta_k}{\sin^3(\theta_k/2)} \right)
  \\
   & > &  \dfrac28 \left(  \dfrac{\cos(\pi/n)}{\sin^3(\pi/(2n))} +
   \sum_{k=[m/2]+1}^m  \dfrac{\cos\theta_k}{(\sqrt{2}/2)^3}
   \right)
      >  \dfrac28 \left(  \dfrac{\cos(\pi/n)}{\sin^3(\pi/2n)} - n\frac{\sqrt{2}}2 \right)
   \\
   & \geq &  \dfrac28 \left( \dfrac{1-\pi/(2n)}{(\pi/(2n))^3}-n\frac{\sqrt{2}}2
   \right) = \dfrac28 \left( \left(\dfrac{2n}{\pi}\right)^3 - \left(\dfrac{2n}{\pi}\right)^2-n\frac{\sqrt{2}}2
    \right)
\end{eqnarray*}
where we have used that $n\geq 4$, and $\cos x \geq 1-x/2$ and $\sin(x/2) \leq x/2$ for $x \leq \pi/4$.
Then,
$$8(g_2(1)-g_1(1)) > \dfrac{16}{\pi^3}n^3 - \dfrac{8}{\pi^2}n^2-n\sqrt{2}-3-3\dfrac{n-2}{\sin(\pi/n)}.$$
Taking $x=\pi/n$, it is enough to see that
$$ 16-8x-\pi\sqrt{2}x^2-3x^3 \leq 3x^2 \dfrac{\pi-2x}{\sin x}$$
for $x\in(0,\pi/4]$. The polynomial is decreasing, while the function on the right is increasing, and the inequality holds for $x=\pi/4$.
This concludes the proof.
\end{proof}


\begin{propn}[\ref{propF3}]
Let $F$ be the function defined in (\ref{eq:k2}). Then, for $n=3$, the following statements hold.
\begin{enumerate}
  \item The equation $F(a)=0$ has exactly two positive solutions, denoted by $z_1$ and $z_2$;
  \item $3/8<z_1<7/16$ and $25/16<z_2<13/8$;
  \item $F(a)<0$ for $a\in (z_1,z_2)$ and $F(a)>0$ for $a\in(0,z_1)\cup (z_2,\infty)$.
\end{enumerate}
\end{propn}
\begin{proof}
For $n=3$ the functions $C_{12}$ and $C_{21}$ are
\begin{equation}
C_{12}(a) = \dfrac{2-a}{(1+a^2-a)^{3/2}}+\dfrac1{(1+a)^2}, \quad
C_{21}(a) = \dfrac{2a-1}{(1+a^2-a)^{3/2}}+\dfrac1{(1+a)^2},
\label{eq:funCn3}
\end{equation}
and the function $F$ is
\begin{equation}
  F(a)=\dfrac{\sqrt{3}}3\,a-\dfrac{2a-1}{(1+a^2-a)^{3/2}}-\dfrac1{(1+a)^2}.
\label{eq:funFn3}
\end{equation}
$F(a)=0$ is equivalent to $g_1(a)=g_2(a)$ where
$$ g_1(a)=(1+a^2-a)^{3/2}, \quad \hbox{and} \quad g_2(a)=\dfrac{3(2a-1)(1+a)^2}{\sqrt{3}a^3+2\sqrt{3}a^2+\sqrt{3}a-3}.$$

The function $g_1\in{\cal C}^1$, is positive, decreases in $(0,1/2)$, increases in $(1/2,\infty)$, $g_1(1/2)=3\sqrt{3}/8$ and $\lim_{a\to\infty}g_1(a)=+\infty$.

The function $g_2$ has only one positive zero at $a=1/2$, and using Descarte's rule and a Bolzano's argument, there exists only one pole at $a=c\in(5/8,3/4)$. Furthermore, $g_2(a)<0$ for $a\in(1/2,c)$ and positive otherwise, and $\lim_{a\to+\infty} g_2(a)=6/\sqrt{3}$.
Its derivative is
$$g_2'(a)= 3\sqrt{3}\,\frac { \left( 1+a \right)  \left(a^3+3a^2+(3-6\sqrt{3})a+1 \right) }{ \left( \sqrt {3}a+2\,
\sqrt {3}{a}^{2}+\sqrt {3}{a}^{3}-3 \right) ^{2}}
$$
and, using again Descarte's rule and Bolzano's theorem, it has only two positive zeros, one located in $(0,1/2)$, which is a local maximum, and one in $(23/16,3/2)$, which is a local minimum. We denote by $a_m$ the local minimum.

First we show that there exits only one solution of $g_1(a)=g_2(a)$ in $(0,1/2)$, denoted by $z_1$. This is a direct consequence of the fact that $g_1(0)=g_2(0)=1$, $g_1$ decreases in the interval $(0,1/2)$, $g_2'(0)>0$ and $g_2$ has unique critical point which is a local maximum in the same interval.

Next we show that  there exits only one solution of $g_1(a)=g_2(a)$ in $(c,\infty)$, denoted by $z_2$. Recall that the function $g_2$ decreases for $(c,a_m)$, increases for $a>a_m$, and has a horizontal asymptote when $a\to \infty$, whereas the function $g_1$ tends to infinity. Thus, at least there must exists one solution. Moreover, $g_1'(a)$ is an increasing function, so that $g_1'(a) > g_1'(1)=3/2$ for $a>1$, and
\begin{eqnarray*}
  g_2'(a) &\leq & 3\sqrt{3}\,\dfrac { \left( 1+a \right) (5a^3+(3-6\sqrt{3})a)}{\left( \sqrt {3}a+2\,
\sqrt {3}{a}^{2}+\sqrt {3}{a}^{3} \right) ^{2}}
= \dfrac{\sqrt{3}}{a} \dfrac { \left( 1+a \right) (5a^3-5a)}{\left( (1+a)^4 \right) ^{2}} \\
 &=& \dfrac{5\sqrt{3}}{a} \dfrac{a-1}{(a+1)^2} \leq  \dfrac{5\sqrt{3}}{8}
\end{eqnarray*}
for $a>1$. Therefore, there cannot exist more than one solution to the equation $g_1(a)=g_2(a)$.  See Figure~\ref{fig:plot_prop1}.
\begin{figure}[!ht]
\centering
\includegraphics[scale=0.3]{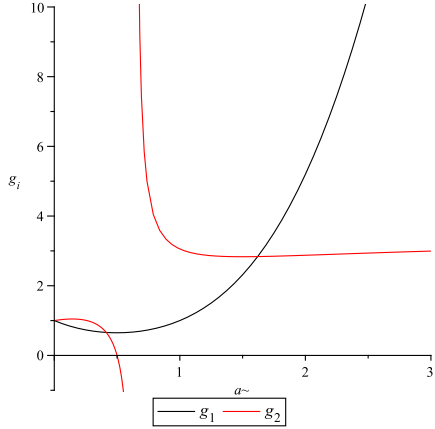}
\caption{Functions $g_1(a)$ and $g_2(a)$ used in the proof of Proposition \ref{propF3}}
\label{fig:plot_prop1}
\end{figure}

It is clear from the above  that  $z_1 < 1 /2$ and $c< z_2$. In fact, using Bolzano's theorem, it can be shown that
$$ z_1\in(3/8,7/16) \quad \hbox{ and } \quad z_2\in(25/16,13/8).$$

Finally, from the fact that $F(1/4)>0$, $F(1)<0$ and $F(2)>0$ we have that $F(a)$ is negative for $z\in(z_1,z_2)$,
and positive otherwise. This concludes the proof.
\end{proof}

\begin{propn}[\ref{propFn}]
Let $F$ be the function defined in (\ref{eq:k2}). Then, for $n\geq 4$, the following statements hold.
\begin{enumerate}
  \item The equation $F(a)=0$ has at least two positive solutions, denoted by $z_1$ and $z_2$.
  \item Suppose that $z_1$ and $z_2$ are the only positive solutions of $F(a)=0$. Then
  \begin{enumerate}
    \item $z_1 < 1 < z_2$ and $z_1z_2 >1$.
    \item $F(a)<0$ for $a\in(z_1,z_2)$, and $F(a)>0$ for $a\in(0,z_1)\cup(z_2,\infty)$.
  \end{enumerate}
\end{enumerate}
\end{propn}
\begin{proof}
The first statement follows from Lemma~\ref{lemFn} and Bolzano's Theorem applied in the intervals $[\epsilon,1]$ and $[1,1/\epsilon]$ for $\epsilon$ small enough. Furthermore, it is clear that $z_1<1<z_2$, and supposing that they are the
only positive solutions of $F(a)=0$, we get easily that $F(a)<0$ for $a\in(z_1,z_2)$, and $F(a)>0$ for $a\in(0,z_1)\cup(z_2,\infty)$. We remark that this fact was already known by Roberts, and proved in \cite{PhDRoberts}.

From Lemma~\ref{lemZn} the property $z_1z_2>1$ follows immediately for $n\geq 5$.
For $n=4$, the result of that lemma is not true (it can be checked numerically).
We claim that, when $n=4$, $F(2/3)>0$ and $F(3/2)<0$. 
Using the fact that $F(1)<0$ and $\lim_{a\to+\infty}F(a) = +\infty$, we get that $ 2/3 < z_1 < 1$ and $z_2>3/2$, so in particular $z_1z_2>1$.

The proof of the claim relies on the fact that for $n=4$ we can write $F(a)=f_1(a)-f_2(a)-f_3(a)$ where
$$ f_1(a)=\left(\frac14 +\frac{\sqrt{2}}2\right)a, \qquad
f_2(a)=\frac{2a-\sqrt{2}}{(1+a^2-a\sqrt{2})^{3/2}}, \qquad
f_3(a)=\frac{2a+\sqrt{2}}{(1+a^2+a\sqrt{2})^{3/2}}.$$
The following equalities and bounds are straightforward computations from which we obtain that $F(2/3)>0$ and $F(3/2)<0$.
\begin{align*}
  f_1(2/3) &= \frac16(1+2\sqrt{2}) > \frac35, & f_1(3/2) &= \frac38(1+2\sqrt{2}) < \frac32, \\
  f_2(2/3) &= -9\frac{3\sqrt{2}-4}{(13-6\sqrt{2})^{3/2}} < \frac{-1}5, &
   f_2(3/2) &= 8\frac{3-\sqrt{2}}{(13-6\sqrt{2})^{3/2}} > \frac{13}{10},
   \\
  f_3(2/3) &= 9\frac{3\sqrt{2}+4}{(13+6\sqrt{2})^{3/2}} < \frac{4}5, &
   f_3(3/2) &= 8\frac{3+\sqrt{2}}{(13+6\sqrt{2})^{3/2}} > \frac{3}{10}.
\end{align*}
These computations conclude the proof of the proposition.
\end{proof}

\section{Acknowledgments}

Both authors would like to thank Gareth Roberts for fruitful discussions and suggestions. Conversations with Richard Moeckel during his stay at CRM are greatly appreciated. 

E.Barrab\'es has been supported by the Spanish grant MTM2013-41168, and the Catalan grant 2014SGR1145. J.M. Cors has been supported by MINECO/FEDER grant numbers MTM2013-40998 and by AGAUR grant number 2014SGR568.

\end{document}